\documentclass[a4paper,12pt]{article}
\usepackage[dvips]{graphicx}
\usepackage{latexsym}
\usepackage{amsmath,amssymb}
\usepackage[mathscr]{eucal}
\usepackage{epsfig}
\usepackage{lscape}

\newcommand{\dis}{\displaystyle}
\newtheorem{Theorem}{Theorem}[section]

\DeclareRobustCommand{\qed}{%
\ifmmode 
\else \leavevmode\unskip\penalty9999 \hbox{}\nobreak\hfill \fi
\quad\hbox{\qedsymbol}}
\newcommand{\openbox}{\leavevmode
\hbox to.77778em{%
\hfil\vrule
\vbox to.675em{\hrule width.6em\vfil\hrule}%
\vrule\hfil}}
\newcommand{\qedsymbol}{\openbox}
\newcommand{\proofname}{Proof}
\newenvironment{proof}[1][\proofname]{\par\normalfont \trivlist \item[\hskip\labelsep   \itshape #1. ]
\ignorespaces
}{%
\qed\endtrivlist }
\begin{document}

\title{Zeros of linear combinations of Laguerre polynomials from different sequences}

\author{Kathy Driver\thanks{Department of Mathematics and Applied
Mathematics, University of Cape Town, Private Bag X3, Rondebosch
7701, Cape Town, South Africa.  Research by this author is supported
by the National Research Foundation of South Africa under grant
number 2053730.} \and Kerstin Jordaan\thanks{Department of
Mathematics and Applied Mathematics, University of Pretoria,
Pretoria, 0002, South Africa. Research by this author is supported
by the National Research Foundation of South Africa under grant
number 2054423.}}
\date{}
\maketitle
\smallskip

\begin{center}
\end{center}
\medskip

\begin{abstract} We study interlacing properties of the zeros of
two types of linear combinations of Laguerre polynomials with
different parameters, namely $R_n=L_n^{\alpha}+aL_{n}^{\alpha'}$ and
$S_n=L_n^{\alpha}+bL_{n-1}^{\alpha'}$. Proofs and numerical
counterexamples are given in situations where the zeros of $R_n$,
and $S_n$, respectively, interlace (or do not in general) with the
zeros of $L_k^{\alpha}$, $L_k^{\alpha'}$, $k=n$ or $n-1$. The
results we prove hold for continuous, as well as integral, shifts of
the parameter $\alpha$.
\end{abstract}

\bigskip
\noindent AMS MOS Classification:\quad 33C45, 42C05.

\medskip

\newpage

\section{Introduction}

\medskip
Let $\mu$ be a positive Borel measure supported on a finite or
infinite interval $[a, b]$ and let $\{p_n\}^\infty_{n=0}$ be the
sequence of polynomials, uniquely determined up to normalization,
orthogonal with respect to $\mu.$ Then $\dis\int^b_a x^k p_n (x) d
\mu (x) = 0$ for $k = 0, 1, \ldots, n - 1$ and it is well known that
the zeros of $p_n$ are real and simple and lie in $(a, b)$.
Moreover, if $a < x_1 < x_2 < \ldots < x_n < b$ and $a < y_1 < y_2 <
\ldots < y_{n-1} < b$ are the zeros of $p_n$ and $p_{n-1}$
respectively, then

$$a < x_1 < y_1 < x_2 < y_2 < \ldots < x_{n-1} < y_{n-1} < x_n <
b,$$ a property usually called the interlacing of the zeros of $p_n$
and $p_{n-1}.$

\medskip
The interlacing of zeros of polynomials is particularly important in
numerical quadrature (cf. \cite{Shohat}) and also arises, inter
alia, in the context of extremal Zolotarev-Markov problems (cf.
\cite{Grinshpun}), the Korous-Peebles problem (cf. \cite{Uvarov})
and the Gelfond interpolation problem (cf. \cite{Gelfond}).

\medskip
In \cite{Marcellan}, Alfaro, Marcell\'{a}n, Pe\~{n}a and Rezola
derived necessary and sufficient conditions for the orthogonality of
$\{Q_n\}^\infty_{n=0},$ where $Q_n (x) = p_n (x) + a_1 p_{n-1} (x) +
\dots + a_k p_{n-k} (x),~ a_k \neq 0,~ n \geq k$ and
$\{p_n\}^\infty_{n=0}$ is a sequence of monic orthogonal
polynomials. Their work extends the results of Peherstorfer (cf.
\cite{Peters}) who established sufficient conditions, when
supp$(\mu) = (- 1, 1)$, on the real numbers $\{a_j\}^k_{j=1}$ such
that $p_n + a_1 p_{n-1} + \ldots + a_k p_{n-k}$ has  $n$ simple
zeros in $(-1, 1).$ Marcell\'{a}n raised a more general question at
OPSFA 2007: Given two different orthogonal sequences
$\{p_n\}^\infty_{n=0}$ and $\{q_n\}^\infty_{n=0}$, under what
conditions does a linear combination $r_n = p_n + a q_n,~ a \neq 0,$
form an orthogonal sequence $\{r_n\}^{\infty}_{n=0}$? A related
question, relevant for applications, is whether and when the zeros
of $r_n$ interlace with the zeros of $p_n,~p_{n-1},~q_n$ or
$q_{n-1}$. One starting point for answering these general questions
is to consider linear combinations of classical orthogonal
polynomials from different sequences but from the same family.

\medskip
In this paper, we consider linear combinations of Laguerre
polynomials $L_n^\alpha$ of the form
$R_n^{\alpha,t}=L_n^{\alpha}+a~L_n^{\alpha+t}$ and
$S_n^{\alpha,t}=L_n^{\alpha}+b~L_{n-1}^{\alpha+t}$ where $\alpha
>-1$, $t>0$ and $a,~b\neq0$. We recall that the Laguerre polynomials
(cf. \cite{Szego}) are orthogonal with respect to the weight
function $e^{-x}x^{\alpha},~\alpha>-1$, on the interval
$(0,\infty)$.

\medskip For $0<t\leq2$, we give proofs (or counterexamples) for the
interlacing of the zeros of $R_n^{\alpha,t}$ and $S_n^{\alpha,t}$
with the zeros of $L_n^{\alpha},~ L_n^{\alpha+t}, ~L_{n-1}^{\alpha}$
and $L_{n-1}^{\alpha+t}$.

\medskip
We will make use of two well known identities (cf. \cite{AbSt},
22.7.30 and 22.7.29)
\begin{eqnarray}\label{AbSt1} L_n^{\alpha} &=&
L_n^{\alpha+1} - L_{n-1}^{\alpha+1}
\\
\label{AbSt2}~\mbox{and}~ xL_n^{\alpha+1}&=&(x-n)L_n^\alpha
+(\alpha+n)L_{n-1}^\alpha\end{eqnarray}

\section{Linear combinations of Laguerre polynomials of the same degree}

Let
\begin{equation}\label{R}R_n^{\alpha,t}=L_n^{\alpha}+aL_n^{\alpha+t},~a\neq0,~\alpha>-1.\end{equation}
\begin{Theorem}
For $0<t\leq2$, the zeros of $R_n^{\alpha,t}$ interlace with the
zeros of (i)$L_n^{\alpha}$, (ii) $L_n^{\alpha+t}$.

\end{Theorem}
\begin{proof}
We know from \cite{DrJo}, Theorem 2.3 that the zeros of $L_n^\alpha$
interlace with the zeros of $L_n^{\alpha+t}$ for $0 <t\leq2$ which
implies that $L_n^{\alpha}$ has a different sign at successive zeros
of $L_n^{\alpha+t}$ and vice versa. Evaluating (\ref{R}) at
succesive zeros $x_i$ and $x_{i+1}$ of $L_n^{\alpha}$ we obtain
\begin{eqnarray*}R_n^{\alpha,t}(x_i)R_n^{\alpha,t}(x_{i+1})&=&a^2L_n^{\alpha+t}(x_i)L_n^{\alpha+t}(x_{i+1})~,
1=1,2,\dots.n-1\\&<&0~\mbox{for all}~a\neq0.\end{eqnarray*}
Therefore $R_n^{\alpha,t}$ has a different sign at successive zeros
of $L_n^\alpha$ and so the zeros interlace. The same argument shows
that the zeros of $R_n^{\alpha,t}$ interlace with those of
$L_n^{\alpha+t}$ by evaluating (\ref{R}) at successive zeros of
$L_n^{\alpha+t}$.
\end{proof}
{\bf Remark:} For the integer values $t=1$ and $t=2$,
$R_n^{\alpha,1}$ and $R_n^{\alpha,2}$ are in fact each a linear
combination of orthogonal polynomials from the same sequence.
Indeed, using the identity (\ref{AbSt1}), we see that
\begin{equation}\label{R1}
R_n^{\alpha,1} = (a+1)L_n^{\alpha+1} -
L_{n-1}^{\alpha+1},\end{equation} and the restrictions on $a$ to
ensure that $\{R_n^{\alpha,1}\}_{n=0}^{\infty}$ has all its zeros in
$(0,\infty)$ can be deduced from (\cite{Brez}, Theorem 3(v)).
Similarly, applying (\ref{AbSt1}) iteratively, we obtain
\begin{equation}\label{R2} R_n^{\alpha,2}= (a+1)L_n^{\alpha+2}-
2L_{n-1}^{\alpha+2}+L_{n-2}^{\alpha+2}\end{equation} and the zeros
of this type of linear combination are discussed in (\cite{Brez},
Theorem 5, and \cite{Shohat}).

\medskip
Evaluating (\ref{R1}) at successive zeros of $L_{n-1}^{\alpha+1}$ ,
one can also prove that the zeros of $R_n^{\alpha,1}$ interlace with
the zeros of $L_{n-1}^{\alpha+1}$. However, the zeros of
$R_n^{\alpha,t}$ do not interlace with the zeros of
$L_{n-1}^{\alpha}$ even in the simple special case when $t=1$, as
illustrated by the following example: For $n=5$, $a=2.33$,
$\alpha=1.45$ and $t=1$, the zeros of $L_{4}^{1.45}$ are
\[\{0.954365,~2.94834,~6.26071,~11.6366\}\] while those of
$R_5^{1.45,1}$ are
\[\{1.17057,~3.01797,~5.80288,~9.83574,~15.9213\}.\]

\section{Linear combinations of Laguerre polynomials of different degree}

Next we consider linear combinations of the type
\begin{equation}\label{S}S_n^{\alpha,t}=L_n^{\alpha}+b~L_{n-1}^{\alpha+t},~b\neq0,~\alpha>-1.\end{equation}
We will need information on the interlacing properties of the two
polynomials $L_n^{\alpha}$ and $L_{n-1}^{\alpha+t}$ in the linear
combination.

\begin{Theorem}\label{L} Let $\alpha
> -1$ and let
\begin{eqnarray*} 0< x_1 < x_2 <\ldots< x_n&\mbox{be the zeros
of}&L_n^\alpha ,\\
0<y_1 < y_2 <\dots< y_{n-1}&\mbox{be the zeros of}&L_{n-1}^{\alpha},\\
0<t_1<t_2<\ldots<t_{n-1}&\mbox{be the zeros of} &L_{n-1}^{\alpha+t}\mbox{ and }\\
0 < X_1 < X_2 <\ldots< X_{n-1}&\mbox{be the zeros
of}&L_{n-1}^{\alpha +2}\end{eqnarray*}  where $0<t<2$. Then
\[0<x_1<y_1<t_1 <
X_1<x_2<\ldots<x_{n-1}<y_{n-1}<t_{n-1}<X_{n-1}<x_n.\]
\end{Theorem}

\begin{proof} A simple computation using (\ref{AbSt1}) and
(\ref{AbSt2}) leads to
\begin{equation}\label{lc}(\alpha+1)L_n^{\alpha+1}(x) =
(\alpha+n+1)L_n^{\alpha} (x) + x L_{n-1}^{\alpha+2}
(x).\end{equation} Evaluating (\ref{lc}) at successive zeros $x_k$
and $x_{k+1}$ of $L_n^\alpha(x)$, we obtain
\begin{eqnarray*}x_kx_{k+1}L_{n-1}^{\alpha+2}(x_k)
L_{n-1}^{\alpha+2} (x_{k+1})& = &(\alpha+1)^2L_{n}^{\alpha+1} (x_k)
L_{n}^{\alpha+1} (x_{k+1}).\end{eqnarray*} The expression on the
right is negative since the zeros of $L_n^{\alpha}$ and
$L_n^{\alpha+1}$ interlace (cf. \cite{DrJo}, Theorem 2.3) and
therefore
\[ x_k<X_{k}<x_{k+1}\mbox{ for each fixed }k,~~k=1,\ldots,n-1.\]
The zeros of $L_{n-1}^{\alpha}$ increase as $\alpha$ increases (cf.
\cite{Szego}, p. 122), hence
\[y_k<t_k<X_k \mbox{ for each fixed }k,~~ k=1,\ldots,n.\]
Finally, since the zeros of $L_n^{\alpha}$ and $L_{n-1}^{\alpha}$
separate each other, we know that \[x_k<y_k<x_{k+1}\mbox{ for each
fixed }k,~~k=1,\ldots,n-1\] and this completes the proof.
\end{proof}

Note that this result extends Theorem 2.4 in \cite{DrJo} to the case
of polynomials of different degree with continuously varying
parameters.

\begin{Theorem}
For $0<t\leq2$, the zeros of $S_n^{\alpha,t}$ interlace with the
zeros of (i)$L_n^{\alpha}$,~~ (ii) $L_{n-1}^{\alpha+t}$.
\end{Theorem}
\begin{proof}
We know from Theorem \ref{L} that the zeros of $L_n^\alpha$
interlace with the zeros of $L_{n-1}^{\alpha+t}$ for $0 <t\leq2$
which implies that $L_{n-1}^{\alpha+t}$ has a different sign at
successive zeros of $L_{n}^{\alpha}$ and vice versa. Evaluating
(\ref{S}) at successive zeros $x_i$ and $x_{i+1}$ of $L_n^{\alpha}$
we obtain
\begin{eqnarray*}S_n^{\alpha,t}(x_i)S_n^{\alpha,t}(x_{i+1})&=&b^2L_{n-1}^{\alpha+t}(x_i)L_{n-1}^{\alpha+t}(x_{i+1})~,
i=1,2,\dots.n-1\\&<&0~\mbox{for all}~b\neq0.\end{eqnarray*}
Therefore $S_n^{\alpha,t}$ has a different sign at successive zeros
of $L_n^\alpha$ and so the zeros interlace. The same argument shows
that the zeros of $S_n^{\alpha,t}$ interlace with those of
$L_{n-1}^{\alpha+t}$ by evaluating (\ref{S}) at successive zeros of
$L_{n-1}^{\alpha+t}$.
\end{proof}

It is interesting to note that in the case of linear combinations of
Laguerre polynomials of different degree, the zeros of
$S_n^{\alpha,t}$ do not interlace with the zeros of
$L_{n-1}^{\alpha}$. Indeed, even in the simplest case when $t=1$ and
$n=5$, $b=2.33$, $\alpha=1.45$ in (\ref{S}), the zeros of
$S_5^{1.45,1}$ are
\[\{1.34638,~3.48132,~6.74108,~11.6384,~20.6928\}\]
while those of $L_{4}^{1.45}$ are
\[\{0.954365,~2.94834,~6.26071,~11.6366\},\] and interlacing does not occur. The zeros of
$S_n^{\alpha,t}$ and $L_n^{\alpha+t}$ are interlacing when $t=1$
since $S_n^{\alpha,1}=L_n^{\alpha+1}+(b-1)L_{n-1}^{\alpha+1}$.
However, when $t=2$, the zeros of $S_5^{1.45,2}$ are
\[\{1.94417,~ 4.47751,~ 8.08954,~ 12.6085,~ 16.7802\}\]
while those of $L_{5}^{1.45+2}$ are
\[\{1.70945,~ 3.92167,~ 7.07942,~ 11.5061,~ 18.0334\},\]
and interlacing fails in this case.

\medskip
\begin{tabbing}
e-mail addresses: \= Kathy.Driver@uct.ac.za\\
\> kjordaan@up.ac.za
\end{tabbing}

\end{document}